\documentclass[conference]{IEEEtran}

\IEEEoverridecommandlockouts

\usepackage{mathtools, cuted}
\usepackage{lipsum, color}

\usepackage{cite}
\usepackage{amsmath,amssymb,amsfonts}
\usepackage{algorithmic}
\usepackage{graphicx}
\usepackage{textcomp}
\usepackage{xcolor}

\usepackage{float}
\usepackage{amsmath}
\usepackage{amssymb}
\usepackage{mathrsfs} 
\newcommand{\R}{\mathbb{R}}
\usepackage{graphicx}

\usepackage{tabularx}

\usepackage{wrapfig}
\usepackage[colorinlistoftodos]{todonotes}
\usepackage{url}
\usepackage{epigraph}
\usepackage{amsthm} 
\usepackage{multicol,lipsum}
\usepackage{etoolbox}
\newtheorem{theorem}{Theorem}
\newtheorem{rem}{Remark}

\newtheorem{lemma}{Lemma}

\usepackage{enumitem}

\def\BibTeX{{\rm B\kern-.05em{\sc i\kern-.025em b}\kern-.08em
    T\kern-.1667em\lower.7ex\hbox{E}\kern-.125emX}}
\begin{document}

\title{A mathematical model of Breast cancer (ER+) with excess estrogen: Mixed treatments using Ketogenic diet, endocrine therapy and Immunotherapy\\

}

\author{\IEEEauthorblockN{ Hassnaa AKIL}
\IEEEauthorblockA{\textit{Laboratory LIPIM} \\
\textit{ENSA Khouribga}\\
\textit{University of sultan moulay slimane }\\
Khouribga, Morocco\\
hassna.akil@usms.ac.ma}
\and
\IEEEauthorblockN{ Nadia IDRISSI FATMI }
\IEEEauthorblockA{\textit{Laboratory LIPIM} \\
\textit{ENSA Khouribga}\\
\textit{University of sultan moulay slimane }\\
Khouribga, Morocco\\
nadidrissi200133@gmail.com}
}

\maketitle

\begin{abstract}
Breast Cancer is a major public health problem and the most common diagnosed malignancy in woman. There have been significant developments in clinical approaches and theoretical experimental to understand the interactions of cancer cells dynamics with the immune system, also developments on analytical and computational models to help provide insights into clinical observations for a better understanding of cancer cells, but more are needed, especially at the genetic and molecular levels mathematically. Treatments such as immunotherapy, chemotherapy, hormone therapy, radiotherapy, and gene therapy are the main strategies in the fight against breast cancer. The present study aims at investigating the effects of estrogen derived from recent models, but this time combined with immunotherapy as a way to treat or inhibit the cancer growth by a mathematical model of breast cancer in situ, governed by a simplified model of nonlinear-coupled ordinary differential equations, that combines important interactions between natural cells, tumor cells, immune cells, ketogenic diet in the presence of an anticancer drug.
Another contribution was to introduce the inhibition effect \textbf{$\epsilon$} for new results and conclusions, A qualitative study was performed and biological interpretations were included to understand the conditions of stability in a realistic way.
\end{abstract}


\section{Introduction}
Female breast cancer was the most common obstructive diseases in women worldwide in 2020 according to the Global Cancer Statistics. In 2020, as reported in CA: A Cancer Journal for Clinicians, an estimated 19.3 million new cancer cases will be diagnosed, with about 10 million people dying from cancer-related causes. Overall, the data revealed that one out of every five men and women in the world will develop cancer over their lifetime. In the future, the analysis estimates that 28.4 million additional cancer cases will be diagnosed globally in 2040, up 47\% from 2020\cite{CA}. The most commonly diagnosed cancer kind in women is breast cancer, with a projected 1.8 million deaths. mainly, rates are high in North America, Australia, New Zealand, and northern and western Europe; lowest in much of Africa and Asia (Torre et al.\cite{torre}). It is also occurring in North Africa, although its incidence is lower than in Western countries (about 1.5 to 2 times lower in women under 50 and 3 to 4 times lower in women over 50) change.

According to The World Health Organization (WHO), in 2020, 2.3 million women worldwide were diagnosed with breast cancer and 685,000 died. By the end of 2020, 7.8 million women had been diagnosed with breast cancer in the past 5 years, making it the most common cancer in the world and the second-largest cancer after lung cancer. Globally, women lose more disability-adjusted life years to breast cancer than any other cancer \cite{0}.

In Morocco there are approximately 50,000 new cancer cases each year. Cancer also accounts for 13.4\% of deaths in the country. Accounting for 20\% of cases. Breast cancer was the most common disease among Moroccan women in 2016, according to a research issued by the Greater Casablanca Cancer Registry of Morocco, accounting for 35.8\% of all new malignancies among women \cite{MR}.

Breast cancer is a disease that affects breast cells and causes unhindered division of mitosis, which can be malignant in breast tissue, it occurs when the cells in the lobules or the ducts become abnormal and divide uncontrollably. These abnormal cells begin to invade the surrounding breast tissue and may eventually spread (Metastasis) via blood vessels and lymphatic channels to the lymph nodes, lungs, bones, brain and liver \cite{meta}.
There are many different types of breast cancer, about 70\% of them are sensitive to the female sex hormone called estrogen. Cells from these cancers have receptor sites that bind to estrogen, thereby promoting their growth and spread \cite{horm}. These cancers are called estrogen receptor-positive cancers (or ER-positive cancers) and it's the type we are discussing in this paper. Cells from tumors are tested to see if they have these receptors. Hormonal (or endocrine) therapy can be used as a treatment in this case such as Tamoxifen \cite{tam}.

However, information on the etiology of this disease remains scarce. For many years Now, there are challenges in researching new ways to understand and fight breast cancer, the causes of breast cancer are not fully understood and not well known yet, so it's hard to say why one woman might develop it and another doesn't. In either way, some risk factors are known to affect the likelihood of developing it. Some of them you can't change, but some you can, after all the three main risk factors for breast cancer namely are, hormonal imbalances, genetical and environmental. Some of the therapies to suppress tumor growth Or inhibit the dynamic of cancer cell are surgery, chemotherapy, radiation therapy, endocrine therapy, targeted therapy and immunotherapy, despite the fact that breast cancer was formerly thought to be difficult to treat with the latter therapy, because to its immunological "coldness," clinical research and novel medications have proven that immunotherapy treatment can enhance breast cancer patient outcomes \cite{imm}. Yet, every treatment has side effects such as hair loss, mood swings, nausea, vomiting, and fatigue, etc...
As with most cancers, the earlier breast cancer is detected and diagnosed, the better the chances of successful treatment.

Over the years, breast cancer modeling has become an invaluable tool for understanding dynamics Behavior of tumor growth during treatment. Some studies show that mathematical modeling helps solve epidemiological problems. Oke et al. \cite{oke} and  Bozkurt Yousef et al.\cite{yousef} Improved Mudufza's model\cite{1}, they have built-in control parameters (ketogenic diet, immune booster and anticancer drugs) hypothetically have interactions between normal and malignant cells. Anyhow, these Studies do not include the concept of a nutritious diet combined with Tamoxifen and immunotherapy in their mathematical models, and even in traditional current treatments used by oncologists and doctors. According to the breast oncologist Dr. Hung Khong, MD \cite{hung}, the majority of research looking at these types of combinations have focused on an other type of cancer: triple-negative breast cancer, however, there is no indication that ER-positive illness is immune to immunotherapy. In fact, there is evidence that ER-positive breast cancer responds to immunotherapy, and it's our main goal using mathematical tools.

This paper is inspired from mathematical models of \cite{1,2,de2009}, but here we introduced the saturation effect of cancer cells and added the immunotherapy drug (ICB Immune checkpoint blockade's type \cite{ICB}) interaction with the immune cells to conclude new results. This work is organized as follow: first section, we construct the mathematical Model of breast dynamics with the presence of cancer treatment, excess estrogen, and a ketogenic diet as a therapeutic tool to help shrink the tumor size \cite{med} (many researches are still needed to see the effect of this diet on cancer patients). The five compartements are modeled by forming a system of differential equations. A qualitative study consists of the existence of equilibria and the study of their local stability, basic reproduction numbers are discussed.

\section{Model description}

In Oke et al.\cite{oke} and de Pillis et al. \cite{de2009} models, and also previous mathematical models published lately, never used the combination of endocrine therapy with a diet and immunotherapy, so we propose a mathematical model to study the dynamics of breast cancer with excess estrogen by considering the saturation effect with the presence of treatments. We present a system of differential equations describing the interactions between normal cells N, tumor cells T , immune cells I, estrogen E and immunotherapy M, then we study the dynamic behavior described by the system bellow:

\begin{equation}
\left\{
\begin{array}{l}
\begin{aligned}
\dfrac{dN}{dt}={} &N(t)(a_1-b_1N(t))-\\
                  &\dfrac{d_1T(t)N(t)}{1+\epsilon T(t)}-l_1N(t)E(t)(1-k),
\end{aligned}\\\\

\begin{aligned}
\dfrac{dT}{dt}={} &T(t)(a_2d-b_2T(t))-g_1I(t)T(t)-m_dT(t)+\\
                  &l_1N(t)E(t)(1-k),
\end{aligned}\\\\
\begin{aligned}
\dfrac{dI}{dt}={} & s+\dfrac{r I(t)T(t)}{o+T(t)}-g_2I(t)T(t)-m I(t)-\\
                  &\dfrac{l_3 I(t)E(t)}{g+E(t)}(1-k)+\dfrac{p_M I(t)M(t)}{j_M+M(t)} ,
\end{aligned}\\\\
\dfrac{dE}{dt}=p(1-k)-\theta E(t),\\\\
\dfrac{dM}{dt}=v_M(t)-n_MM(t)+\dfrac{\chi M(t)I(t)}{\xi+I(t)}.
\end{array} 
\right.
\label{equ}
\end{equation}

We now explain the model parameters and describe the terms biologically for a better understanding of the interactions presented in each equation.
\subsection{Modeling normal cells}
In system (\ref{equ}), normal cells are represented by the first equation, $a_1$ is the logistic growth rate of the normal cells, $b_1$ is the natural death, therefore; it is very reasonable to model the infection rate by a saturated incidence
of the form $\dfrac{d_1TN}{1+\epsilon T}$, where $d_1$ and $\epsilon$ are positive
constants which, respectively describe the inhibition rate of normal cells due to the DNA damage, and the saturation effect. Excess estrogen leads to DNA mutation and thus normal cells population will be reduced and transformed into tumor cells by $l_1NE$ term.
\subsection{Modeling tumor cells}
In the second equation, the first term is the limited growth term for tumor cells that depends on the ketogenic diet doze $d$, $m_d$ is the death rate of tumor cells as a result of starvation of nutrients, glucose, note that keto diet \cite{4} put the body in ketosis that is a metabolism that turns fat into ketones, which helps protect against some cancers. Tumor cells will be removed due to the immune response by $g_1$.
\subsection{Modeling Lymphocytes(Immune cells)}
Lymphocytes, including T-cells, T-regulatory cells, and natural killer cells (NK), and their cytokine release patterns are implicated in both primary prevention and recurrence of breast cancer \cite{5}. In the third equation $s$ is the constant source rate of immune response, the presence of tumor cells stimulates and activate the immune response; resulting  in growth of immune cells, this is represented by a positive nonlinear growth term also called the Michaelis–Menten interaction term for immune cells $\dfrac{r IT}{o+T}$, where $r$ is the immune response rate, and $o$ is the immune threshold rate, $g_2$ represents the inactivation of immune cells due tumour cells as the interaction coefficient, $m$ is the natural death rate of immune cell, the next term is the limited rate suppression of $I$ due to excess estrogen, where $l_3$ is the suppression rate, and $g$ is the estrogen threshold rate, the last term represente a Michaelis-Menten interaction, it's the activation of immune cells by immunotherapy.
\subsection{Modeling estrogen}
The last equation represents the estrogen dynamics, this hormone plays various roles in the reproductive system of a female, many Studies show higher blood levels of the estrogen called estradiol increase the risk of breast cancer in postmenopausal women \cite{3}, here $p$ is the source rate of estrogen, and $\theta$ is the decay rate after being washed out from the body.
\subsection{Immunotherapy}
We differ our work from previous models such as Oke et al. by adding this equation gonverning the amount of the drug injected per day per litre of body volume, its turnover represented by $n_MM$, and finally a Michaelis–Menten term in the drug used $\frac{\chi M(t)I(t)}{\xi+I(t)}$, representing the production of immunotherapy from activated immune cells ($CD8^{+}T$).

\section{Model analysis}
\subsection{Positiveness and existence of equilibria: }
\begin{theorem}
System~\eqref{equ} has one unique solution $(N,T,I,E,M)$ in $\R_{+}^5$.
\end{theorem}
\begin{proof}
See \cite{murray}.
\end{proof}
\begin{rem}
The aim of our model is to investigate cellular populations, all variables and parameters are non-negative. Based on the biological findings, the system of Equation~\eqref{equ} will be evaluated in $\Omega$ discused in the next theorem, that guarantees that the system is well-posed in such a way that solutions with non-negative initial conditions persist non-negative for all $t>0$, making the variables biologically meaningful. As a result, we get the theorem below.
\end{rem}

\begin{theorem}
The following set, 

$$\Omega=\left\{(N, T, I, E, M) \in \R_{+}^5\right\},$$
is a positively invariant set of model~\eqref{equ}.
\end{theorem}

\begin{proof}
Let $P(t)=(N(t),T(t),I(t),E(t),M(t))$ be any positive solution of the model, with initial condition $P(0)=(N(0),T(0),I(0),E(0),M(0)) \in \R_{+}^{5},$ from the model~\eqref{equ} we get:
$$
\begin{aligned}
&N(t)= N(0)  e^{\int_{0}^{t}a_1-b_1N(\tau)-\frac{d_1T(\tau)}{1+\epsilon T(\tau)}-l_1E(\tau)(1-k) \mathrm{d} \tau},\\
&T(t)=T(0)  e^{\int_{0}^{t}a_2d-b_2T(\tau)-g_1I(\tau)-m_d+\frac{l_1N(\tau)E(\tau)}{T(\tau)(1-k)} \mathrm{d} \tau}, \\
&I(t)=I(0)  e^{\int_{0}^{t}\frac{s}{I(\tau)}+\frac{rT(\tau)}{o+T(\tau)}-g_2T(\tau)-m - \frac{l_3E(\tau)}{g+E(\tau)}(1-k)+\frac{p_M M(\tau)}{j_M+M(\tau)} \mathrm{d} \tau},\\
&E(t)=E(0)  e^{\int_{0}^{t} \frac{p(1-k)}{E(\tau)}-\theta \mathrm{d} \tau}, \\ 
&M(t)=M(0)  e^{\int_{0}^{t} \frac{v_M}{M(\tau)}-n_M+\frac{\chi I(\tau)}{\xi+I(\tau)}\mathrm{d} \tau}.
\end{aligned}
$$
By starting from a positive initial condition each solution of the model remain positive for all $t \geq 0$, thus $\Omega$ is a positively invariant set of our model.\par
From the first and second equation of $N(t)$ and $T(t)$, and by using the comparison theorem we obtain:
$$
\frac{\mathrm{d} N(t)}{\mathrm{d} t} \leq N(t)(a_1-b_1N(t)),$$
Thus $$\limsup _{t \rightarrow+\infty} N(t) \leq \frac{a_1}{b_1} .
$$
and,
$$\limsup _{t \rightarrow+\infty} T(t) \leq \frac{a_2}{b_2} .
$$

Consequently, it can be shown that $I, E, M$ are also bounded, which completes the proof.
\end{proof}

\subsection{Equilibria}

In this section we study the existence and stability of the equilibria, which also represents the critical points of system~\eqref{equ}. The model system admits seven steady states in which there are one tumor-free equilibrium point, five dead equilibria, and finally the coexisting equilibria discussed in a brief way.

Let $\bar{P}=(\bar{N},\bar{T},\bar{I},\bar{E},\bar{M})$ be an equilibrium point, We have $\bar{N}$, $\bar{T}$, $\bar{I}$, $\bar{E}$ and $\bar{M}$ are positive variables since cell populations are non-negative and real. Therefore, all parameters considered in this model are positive.
\subsubsection{Tumor free equilibrium point}
$$P_{0}=\left(\bar{N}_{0}, 0, \bar{I}_{0}, \bar{E}_{0}, \bar{M}_{0}\right),$$
where,
\begin{itemize}
\item $\bar{N}_{0}=\dfrac{a_{1}-l_{1} \bar{E}_{0}(1-k)}{b_{1}},$
\item$\bar{I}_{0}=\dfrac{s}{m_{+} \dfrac{l_{3} \bar{E}_{0}}{g+\bar{E}_{0}}(1-k)-\dfrac{p_{M} \bar{M}_{0}}{j+\bar{M}_{0}}},$
\item $\bar{E}_{0}=\dfrac{(1-k) p}{\theta},$
\item $\bar{M}_{0}=\dfrac{v_{M}}{n_{M}-\dfrac{\chi \overline{I_{0}}}{\xi_{+} \overline{I_{0}}}}.$
\end{itemize}
\begin{rem}
$CD8^{+}T$ cells also called cytotoxic T lymphocytes, are a type of immune cells that have the  capacity to react to pathogens like infections, virus and cancer \cite{CD8}.\\
If we consider $CD8^{+}T$ as the only type of immune cells for our model, then during the tumor-free state, $I(t)$ will be zero for all $t$, since the activation of these effector cells depends on tumor.
\end{rem}
\begin{lemma}

$P_{0}$ exists if and only if :
\begin{enumerate}[label=\roman*)]
    \item $\bar{E}_{0}\leq\dfrac{a_1}{l_1(1-k)}, with\quad k<1,$ \\
    \item $\bar{M}_{0}\leq\dfrac{m}{p_{M}},$ \\
    \item $\bar{I}_{0}\leq\dfrac{n_{M} \xi}{\chi-n_{M}},$\\
    \item $l_1\leq\dfrac{\theta a}{p(1-k)^{2}}.$
\end{enumerate}
\end{lemma}

Biologically, it means that the growth rate of normal cells must be more than the estrogen amount, also we notice that the $\bar{I}_0$ in this case depends on the suppression of estrogen and the amount of immunotherapy dose, unlike the the estrogen-free model studied by T. Sundaresan et al. \cite{2}, where it depends only on the nature of the dynamics.

\subsubsection{Death equilibrium points type 1}~\par

\vskip1mm
Dead equilibria are states when both normal cells and
tumor cells have died off, maybe by a mastectomy surgery \cite{surg}, or death, and it's given by:
$$
P_{di}=\left(0,0, \bar{I}_{i}, \frac{(1-k) p}{\theta}, \bar{M}_{i}\right) ,
$$
such that,
$$\bar{M_{i}}=\dfrac{v_{M}}{n_{M}-\dfrac{\chi \bar{I_{i}}}{\xi_{i}+ \bar{I_{i}}}} \text { for  } i\in\{1,2\},$$
and $\bar{I}_{i}$ verifies this equality:
\begin{equation}
\bar{I}^{2}\left(A \chi+p_{M} v_{M}\right)-\bar{I}\left(A n_{M}-p_{M} v_{M} \xi+s \chi \right)+s n_{M}=0,
\label{equI}
\end{equation}
such that,
$$A=m+\frac{l_{3} \bar{E}}{g+\bar{E}}(1-k),$$
therefore, the discriminant of (2) is:
\begin{equation}
\begin{aligned}
&\Delta=v_{M}(t)^{2} \chi^{2} p_{M}^{2}-2 A v_{M}(t) \chi n_{M} p_{M}-2 v_{M}(t) \chi s \xi p_{M}\left(A n_{M}\right)^{2}\\
&-2 A s \xi n_{M}-4 v_{M}(t) s n_{M} p_{M}+s^{2} \xi^{2},
\end{aligned}
\label{delta}
\end{equation}
we set $B=A\chi+p_{M} v_{M},$\\
 $B\geq0 \Rightarrow \Delta\geq0$, which means that equation (\ref{equI}) has two distinct real roots:
\begin{equation}
\bar{I}_{1,2}=\dfrac{-p_{M} v_{M}(t) \chi+A n_{M}+s \xi_{\mp} \sqrt{\Delta}}{2\left(A \xi+p_{M} v_{M}(t)\right)}.
\label{rootsI}
\end{equation}

We summerize then the existence of the two dead equilibrium points in the following lemma,
\begin{lemma}
$P_{di}$ for $i=1,2.$ exists if and only if :
\begin{enumerate}[label=\roman*)]
    \item $\dfrac{p_{M} v_{M} \xi}{s \chi A n_{M}}<1,$
    \item $k<1,$ 
    \item $ n_M\geq \dfrac{\chi\bar{I}_{i}}{\xi+\bar{I}_{i}}.$
\end{enumerate}
\end{lemma}
Biologically, it means that the existence of these equilibria depends on the natural elimination or excretion of IL-2, which must be more than its production from the activated $CD8^{+}T$.

\subsubsection{Death equilibrium points type 2}

$$
P_{di}=\left(0, \bar{T}_{i}, \frac{a_2 d-b_{2} \bar{T}_{i}-m_{d}}{g_{1}}, \frac{(1-k)p}{\theta}, \frac{v_{M}}{n_{M}-\frac{\chi \bar{I}_{i}}{\xi+\bar{I}_i}}\right),
$$
for $i={3,4,5}.$

To find the expressions of $\bar{T_i}$, we solve the above quadratic equation:
\begin{equation}
\begin{aligned}
\nu(\bar{T})=&\bar{T}^3\left(b_{2} g_{2}\right.)+\bar{T}^{2}(b_{2} m+b_{2} g_{2}o-b_{2} r+b_{2}C\\
&-a_{2} d g_{2}-a_{2} d C +m_d g_{2})+\bar{T}(b_{2} m o+b_{2} C o\\
&-a_{2} d m-a_{2} d g_{2} o+r a_2d-a_{2}do C+g_{1} s\\
&+m_{d} m+m_{d} g_{2} \theta-r m_{d}+m_{d} C) \\
&+m_{d} C o+m_{d} m o+g_{1} s o-a_2d m o,
\end{aligned}
\label{Tequation}
\end{equation}

where,
$$
C=\frac{l_{3} \bar{E}}{g+\bar{E}}(1-k)-\frac{p_{M} \bar{M}}{j+\bar{M}},
$$
thus,
$$\bar{T}_i=Roots (\nu(\bar{T}))\quad \text{for}\quad i=3,4,5.$$
We finally can summarize the expression of the death equilibrium points of type 2 as follows:
$$
P_{d3, d4, d5}=\left(0, \bar{T}_{3,4,5}, \bar{I}_{3,4,5} ,\frac{(1-k) p}{\theta}, \bar{M}_{3,4,5}\right).
$$
\begin{lemma}
The dead equilibria of type 2 exist if:
\begin{enumerate}[label=\roman*)]
    \item $k<1,$
    \item $ \dfrac{a_{2} d-m_{d}}{b_{2}}-\dfrac{g_{1}^{2} n_{M}}{b_{2}\left(\chi-g_{1}n_M\right)} \leqslant \bar{T}_{3,4,5} \leqslant \dfrac{a_{2} d-m_{d}}{b_{2}}.$
\end{enumerate}
\end{lemma}
\vspace{1em}
\subsubsection{Co-existing point}~\par
In this case all cell populations survive the competition and they coexist, it is given by:
$$
P_e=\left(N_{e}, T_{e}, I_{e}, E_{e},M_{e}\right),
$$
such that,
\begin{itemize}
\item $N_e=\dfrac{1}{b_1}\left(a_1-\dfrac{d_1T_e}{1+\epsilon T_e}-l_1E_e(1-k)\right)=\psi_{1}(I_e),$

\item $T_e= Roots(\mu(T_e))=\psi_{2}(I_e),$

\item $I_e=\dfrac{s}{g_2T_e-\dfrac{r T_e}{o+T_e}+m+\dfrac{l_3E_e}{g+E_e}(1-k)- \dfrac{p_{M} M_e}{j+M_e}},$

\item $E_e=\dfrac{(1-k) p}{\theta},$
\item $M_e=\dfrac{v_{M}}{n_{M}-\dfrac{\chi I_e}{\xi+ I_e}}=\psi_{3}(I_e),$
\end{itemize}

where,
\begin{equation}
\mu(T_e)=b_2T_e^{2}+(g_1I_e+m_d-a_2d)T_e-l_1N_eE_e(1-k),
\label{6}
\end{equation}
we let,
\begin{align*}
a&=b_2,\\
b&=g_1I_e+m_d-a_2d,\\
c&=l_1N_e E_e(1-k),
\end{align*}

the existence of the coexisting point depends on the sign of (\ref{6}) roots, we discuss three cases using its coefficients, when $b$ and $c$ are positive we get two negative roots, means there exist no realistic equilibria in this case, therefore, for $c<0$ we  get one equilibrium point, and for $b$ either positive, or negative we obtain two roots with opposite signs, which means there exist at least one coexisting point, finally the case where $b<0$ and $c>0$ will allow us to get two coexisting points besides the other equilibria we found in our previous parts.

We summarize the existence of the coexisting point in the following lemma,
\begin{lemma}
$P_e$ exists if and only if:
$$I_{e} \in\left]0,\dfrac{a_2d-m_d}{g_1}\right[\text{and}\quad  k<1.$$
\end{lemma}


\section{Stability analysis}
In this section we analyze the equilibria in terms of their local stability by
means of eigenvalues to identify conditions that can help eliminate tumor cells, We apply the Hartman Grobman theorem \cite{perko}.
Let $\bar{P}=(\bar{N},\bar{T},\bar{I},\bar{E},\bar{M})$ be an arbitrary equilibrium of model (\ref{equ}).

Hence, the associated characteristic equation of the jacobian matrix of our system at $P$is given by,
\begin{strip}
\begin{equation}
\left|\begin{array}{ccccc}{a_1-2 b_{1} \bar{N}-\dfrac{d_{1} \bar{T}}{1+e \bar{T}}-l_{1} \bar{E}(1-k)-\lambda } & {\dfrac{d_1 \bar{N}}{(1+e\bar{T})^2}} & {0} & {-l_1 (1-k)\bar{N}}& {0} \\\\

{l_1 (1-k)\bar{E}} & {a_{2} d-2 b_{2} \bar{T}-g_{1} \bar{I}-m_{d}-\lambda  } & {-g_1 \bar{T}} &{l_1 \bar{N} (1-k)} & {0}  \\\\

{0} & {{\dfrac{r o \bar{I}}{(o+\bar{T})^{2}}-g_{2} \bar{I}}} & {J_{33}-\lambda } & { -\dfrac{l_3\bar{I}g}{(g+\bar{E})^2}(1-k) } & {\dfrac{P_{\bar{M}} \bar{I} j_{\bar{M}}}{(j_{\bar{M}}+\bar{M})^2}}  \\\\

{0} & {0} & {0} & {-\theta-\lambda }& {0}\\\\

{0} & {0} & {\dfrac{\chi \bar{M} \xi}{(\xi+\bar{I})^{2}}} & {0}& {-n_{\bar{M}}+\dfrac{\chi \bar{I}}{\xi+\bar{I}}-\lambda } \end{array}\right|=0.
\label{CE}
\end{equation}
\end{strip}
where,\[J_{33}=\dfrac{r \bar{T}}{o+\bar{T}}-g_{2} \bar{T}-m-\dfrac{l_{3} \bar{E}}{g+\bar{E}}(1-k)+\dfrac{P_{\bar{M}}\bar{M}}{j_{\bar{M}}+\bar{M}}. \]
Now, we introduce reproductive numbers and conditions for the stability of tumor-free state, denote:
$$\mathscr{R}_{0}=\dfrac{A_{6} A_{9}}{A_{10} A_{5}},\quad \text{and}\quad \mathscr{R}_{1}=\dfrac{A_{1} A_{2}}{A_{0} A_{3}},$$
where,
\begin{align*}
A_0 &=a_1-2 b_{1}N-l_{1} E(1-k),\\
A_1 &=l_1 E(1-k),\\
A_2 &=d_1 N,\\
A_3 &=a_{2} d-g_{1}  I-m_{d},\\
A_4 &=\dfrac{r  I}{o}-g_{2}  I,\\
A_5 &=-m-\dfrac{l_{3}  E(1-k)}{g+ E}+\dfrac{P_{M}  M}{j_{M}+ M},\\
A_6 &=\dfrac{\chi  M \xi}{(\xi+ I)^{2}},\\
A_7 &=l_1 N(1-k),\\
A_8 &=\dfrac{l_{3}  I g}{g+ E}(1-k),\\
A_9 &=\dfrac{P_{M}  I j_{M}}{(j_{M}+M)^{2}},\\
A_{10}&=-n_M+\dfrac{\chi I}{\xi+I}.
\end{align*}
\begin{theorem}
The tumor-free equilibrium point $P_{0}$ of system (\ref{equ}) is locally asymptotically stable if $\mathscr{R}_{0}<1$ and $\mathscr{R}_{1}<1$, otherwise it is unstable.
\end{theorem}

\begin{proof}
Let $J_{p_0}$ be the jacobian matrix around $P_0$:
$$J_{p_0}=\left(\begin{array}{ccccc}
A_{0} & A_{2} & 0 & -A_{7} & 0 \\
A_{1} & A_{3} & 0 & A_{7} & 0 \\
0 & A_{4} & -A_{5} & -A_{8} & A_{9} \\
0 & 0 & 0 & -\theta & 0 \\
0 & 0 & A_{6} & 0 & A_{10}
\end{array}\right)$$
and the characteristic equation at $P_0$ is given by:\\
\begin{multline*}
\mathbb{P}(\lambda)=(\lambda+\theta)\left(\lambda^{2}-A_{10} \lambda+A_{5} \lambda-A_{10} A_{5}-A_{6} A_{9}\right)\\\left(\lambda^{2}-A_{0} \lambda- A_{3} \lambda-A_{1}A_{2}+A_{0}A_{3}\right),
\end{multline*}
we can clearly see that  this equation has five eigenvalues 
\begin{align}
\lambda_1=-\theta,\\
\lambda^2-\left(A_{10}-A_{5}\right) \lambda-A_{10} A_{5}\left(1- \dfrac{A_{6} A_{9}}{A_{10} A_{5}}\right)=0,\label{9}\\
\lambda^2-\left(A_{0}+A_{3}\right) \lambda+A_{0} A_{3}\left(1- \dfrac{A_{1} A_{2}}{A_{0} A_{3}}\right)=0,\label{10}
\end{align}
equation (\ref{9}) admits two eigenvalues: $\lambda_2$ and $\lambda_3$, to study their nature we apply the Routh hurwitz criteria, we let :
$$
H_1=\left(\begin{array}{ll}
a_{1} & a_{0} \\
a_{3} & a_{2}
\end{array}\right),
$$
where,
\begin{align*}
a_0&=1,\\
a_1&=-(A_{10}-A_{5}),\\
a_2&=-A_{10} A_{5}(1-\mathscr{R}_{0}),\\
a_3&=0.
\end{align*}
The roots of (\ref{9}) has negative real parts only if all the principle diagonal minors of the Hurwitz matrix are positive, provided that:\\
$$a_0>0,\quad \Delta_1=a_1>0,\quad \Delta_2=\left|\begin{array}{ll}a_{1} & a_{0} \\ a_{3} & a_{2}\end{array}\right|>0. $$
Therefore, the eigenvalues of (\ref{9}) are negative only if $\mathscr{R}_{0}<1$.

We repeat the same process for (\ref{10}), we get two eigenvalues: $\lambda_4$, $\lambda_5$ and the other reproductive number $\mathscr{R}_{1}$.
We conclude that our system is stable at the tumor free equilibrium point if and only if :
$$\mathscr{R}_{0}<1\quad \text{and} \quad \mathscr{R}_{1}<1. $$
\end{proof}
\begin{rem}
$\lambda_3$ and $\lambda_5$ are negative if and only if :
$$
\dfrac{sM}{V_{M}}<{I}< \dfrac{a_{2}(1+d)-2 b_{1} N-l_{1} E(1-k)-m_{d}}{g_{1}}.
$$

Biologically its means that the immune response must be greater than the immunotherapy dose taken during the treatment.
\end{rem}
Now we study the stability of the death equilibrium points of type 1 $P_{d1,d2}$, and we let $\mathscr{R}_{IM}=\dfrac{B_{5} B_{7}}{B_{4} B_{8}}$ be the reproduction number for the immune system response such that,
\begin{align*}
B_0 &=a_1-l_{1} E(1-k),\\
B_1 &=l_1 E(1-k),\\
B_2 &=a_{2} d-g_{1}  I-m_{d},\\
B_3 &=\dfrac{r  I}{o}-g_{2}  I,\\
B_4 &=-m-\dfrac{l_{3}  E(1-k)}{g+ E}+\dfrac{P_{M}  M}{j_{M}+ M},\\
B_5 &=\dfrac{\chi  M \xi}{(\xi+ I)^{2}},\\
B_6 &=-\dfrac{l_{3}  I g}{(g+ E)^2}(1-k),\\
B_7 &=\dfrac{P_{M}  I j_{M}}{(j_{M}+M)^{2}},\\
B_8 &=-n_M+\dfrac{\chi I}{\xi+I}.
\end{align*}

Therefore equation (\ref{CE}) becomes:
\begin{equation}
\label{jac2}    
(\lambda+\theta)(B_{4} B_{8}-B_{4} \lambda-B_{5} B_{7}-B_{8} \lambda+\lambda^{2})(B_{2}-\lambda)(B_{0}-\lambda)=0
\end{equation}

\begin{theorem}
Death equilibria of type 1 are stable if and only if:
\begin{itemize}
    \item $\mathscr{R}_{IM}<1,$
    \item $B_0, B_2, B_4, B_8 <0,$
\end{itemize}
\end{theorem}
\begin{proof}
The characteristic equation (\ref{jac2}) has five eigenvalues:
$$
\begin{aligned}
&\lambda_{1}=-\theta,\\
&\lambda_{2,3}=Roots(F(\lambda)=B_{4} B_{8}-B_{4} \lambda-B_{5} B_{7}-B_{8} \lambda+\lambda^{2}), \\
&\lambda_{4}=B_{2}, \\
&\lambda_{5}=B_{0}.
\end{aligned}
$$
For the system to be stable we must have all the eigenvalues with positive real parts. First let's work on the $\lambda_{2,3}$ sign:
$$
\begin{aligned}
F(\lambda) &=B_{4} B_{8}-B_{4} \lambda-B_{8} B_{+}-B_{8} \lambda_{+} \lambda^{2} \\
&=\lambda^{2}-\lambda(B_{4}+B_{8})+B_{4} B_{8}(1-\dfrac{B_{5} B_{7}}{B_{4} B_{8}}),
\end{aligned}
$$
using Routh-Hurwitz criteria, we summarize that the death equilibria of type 1 are stable only if:
\begin{equation}
\left\{\begin{aligned}
& B_{4}<0 \quad \text{and}\quad B_{8}<0, \\
& B_{4}+B_{8}<0, \\
&\mathscr{R}_{IM}<1.
\end{aligned}\right.
\end{equation}

In a more detailed way, we will determine necessary conditions and also their biological interpretations for each part as follow:
$$
B_{4}<0 \iff \dfrac{P_{M}M}{J_{M}+M}<m+\dfrac{l_{3} E(1-k)}{g+E},
$$

biologically means that the activation of the $CD8^+T$ by the immunotherapy must be smaller than the deactivation of the immune response by estrogen.
$$
B_{8}<0 \iff \dfrac{\chi I}{\xi_{1}+I}<n_{M},
$$
here we should see that the production of $IL_2$ from activated immune cells must be less than the excretion of the $IL_2$.
$$
B_{0}<0 \iff a_{1}<l_{1} E(1-k),
$$
this condition shows that the damaged natural cells due to excess estrogen should be more than it's natural growth rate.
and finally,
$$
B_{2}<0 \iff I>\dfrac{a_{2} d-m_{d}}{g_{1}},
$$
in this case the immune response needs to be very strong so it affect and blocks the tumor cells growth.
\end{proof}

Next we study the behavior of the death equilibria of type 2, in this case (\ref{CE}) becomes,

\begin{equation}
\begin{multlined}
\mathbb{P}(\lambda)=(\lambda+\theta)(C_{0}-\lambda)[\lambda^{2}\left(C_{2}+C_{5}+C_{9}\right)-\lambda^{3}+\\
\lambda\left(C_{6} C_{8}-C_{5} C_{9}-C_{3} C_{4}-C_{2} C_{9}-C_{2} C_{5}\right)+\\
C_{2} C_{5} C_{9}-C_{2} C_{6} C_{8}+C_{3} C_{4} C_{9}]=0,
\end{multlined}
\label{charaC}
\end{equation}

where,
$$
\begin{aligned}
&C_{0}=a_{1}-\frac{d_{1} T_{1}}{1+e T}-l_{1} E, \\
&C_{1}=l_{1} E(1-k), \\
&C_{2}=a_{2} d-2 b_{2} T-g_{1} I-m_{d}, \\
&C_{3}=\frac{r I o}{(o+T)^{2}}-g_{2} I, \\
&C_{4}=g_{1} T,\\
&C_{5}=\frac{r T}{o+T}-g_{2} T-m-\dfrac{l_{3} E(1-k)}{g+ E}+\dfrac{P_{M}  M}{j_{M}+ M},\\
&C_{6}=\dfrac{\chi  M \xi}{(\xi+ I)^{2}},\\
&C_{7}=-\dfrac{l_{3}  I g}{(g+ E)^2}(1-k),\\
&C_{8}=\dfrac{P_{M}  I j_{M}}{(j_{M}+M)^{2}},\\
&C_{9}=-n_M+\dfrac{\chi I}{\xi+I}.
\end{aligned}
$$
\begin{theorem}
The death equilibria of type 2 are stable only if:
\begin{enumerate}[label=\roman*)]
    \item $a_1<\dfrac{d_{1} T_{1}}{1+e T}-l_{1} E,$
    \item $C_{6} C_{8}-C_{5} C_{9}-C_{3} C_{4}-C_{2} C_{9}-C_{2} C_{5}>0,$ 
    \item $ C_{2} C_{5} C_{9}-C_{2} C_{6} C_{8}+C_{3} C_{4} C_{9}>0.$
\end{enumerate}
\end{theorem}
\begin{proof}
The characteristic equation (\ref{CE}) has five eigenvalues:
$$
\begin{aligned}
&\lambda_{1}=-\theta,\\
&\lambda_{2}=C_{0},\\
&\lambda_{3,4,5}=Roots(G(\lambda)),
\end{aligned}
$$
such that,

$\begin{multlined}
G(\lambda)=-\lambda^{3}+\lambda^{2}(C_{2}+C_{5}+C_{9})+\\
\lambda(C_{6} C_{8}-C_{5} C_{9}-C_{3} C_{4}-C_{2} C_{9}-C_{2} C_{5})\\
+C_{2} C_{5} C_{9}-C_{2} C_{6} C_{8}+C_{3} C_{4} C_{9},
\end{multlined}$\\
in a short form we let:
$$
f(\lambda)=-a_0\lambda^{3}+a_1\lambda^{2}+a_2\lambda+a_3,
$$
 $\mathfrak{Re}(\lambda_i)<0$ for $i={3,4,5}$ if and only if $a_i>0$ and the principal minors of the Hurwitz matrix of $f$ are positive.

$\lambda_1<0$, $\lambda_2$ has positive real part if:
$$a_1<\frac{d_{1} T_{1}}{1+e T}-l_{1} E,$$

biologically it means that the growth rate of normal cells $a_1$, must be smaller than it's damage rate by excess estrogen and tumor cells effect; to reach one of the conditions of stability where $C_0<0$. 
\end{proof}


Finally we study the stability of our system at $P_e$, let 

$$
P_{e}=\left(\phi_{1}\left(I_{e}\right), \phi_{2}\left(I_{e}\right), I_{e}, \frac{p(1-k)}{\theta}, \phi_{3}\left(I_{e}\right)\right),
$$
such that $\phi_{i}$ for $i={1,2,3}$ are functions of $I_e$.

Repeating the same process as before, we develop our calculus, and from the jacobian matrix at $P_e$ we get the following characteristic equation:

\begin{equation}
\begin{aligned}
&\mathbb{P}(\lambda)=(D_{10}-\lambda)[(D_{0}-\lambda)(D_{3}-\lambda)(D_{6}-\lambda)(D_{12}-\lambda)\\
&-D_{11}D_{7}-D_{5}D_{4}(D_{12}-\lambda)(D_{0}-\lambda)\\
&-D_{2} D_{1}(D_{6}-\lambda)(D_{12}-\lambda)+D_{2} D_{11} D_{7}]=0,
\end{aligned}
\label{pec}
\end{equation}
such that the $D_i$ are the jacobian matrix elements at $P_e$,

the equation (\ref{pec}) admets five eignevalues, using mathematical calculus tools such as Maple, we can define each eignevalues, therefore we can conclude the stability of this system at this point in a brief way in the theorem bellow:

\begin{theorem}
The system is stable at $P_e$ if and only if:
 $\mathfrak{Re}(\lambda_i)<0$ for $i={2,3,4,5}$, while $\lambda_1=-\theta<0.$
\end{theorem}
\begin{rem}
To study the stability of the system at $P_e$, we can also develop its characteristic equation (\ref{pec}), and study the sign of its coefficients using the Routh-Hurwitz criteria using the same previous calculus.
\end{rem}

\vspace{12pt}
\color{red}

\end{document}